\documentclass{article}
\usepackage{eurosym}
\usepackage{amssymb}
\usepackage{amsfonts, color}
\usepackage{graphicx}
\usepackage{amsmath}
\usepackage{verbatim}
\usepackage{enumitem}
\providecommand{\U}[1]{\protect\rule{.1in}{.1in}}
\newtheorem{theorem}{Theorem}[section]

\newtheorem{corollary}[theorem]{Corollary}

\newtheorem{definition}[theorem]{Definition}

\newtheorem{lemma}[theorem]{Lemma}

\newtheorem{proposition}[theorem]{Proposition}
\newtheorem{remark}[theorem]{Remark}

\newenvironment{proof}[1][Proof]{\textbf{#1.} }{\ \rule{0.5em}{0.5em}}

\newcommand{\HH}{\mathbb{H}^2}

\newcommand{\hh}{\mathbb{H}^2}
\newcommand{\R}{\mathbb{R}}
\newcommand{\pO}{\partial\Omega}
\begin{document}

\title{CMC Graphs With Planar Boundary in $\mathbb{H}^{2}\times \mathbb{R}$}
\author{Ari J. Aiolfi, Patr\'icia Klaser}
\maketitle

\begin{abstract}
Let $\Omega \subset \mathbb{R}^{2}$ an unbounded convex domain and $H>0$ be given,
there exists a graph $G\subset \mathbb{R}^{3}$ of constant mean curvature $H$
over $\Omega $ with $\partial G=$ $\partial \Omega $ if and only if $\Omega $
is included in a strip of width $1/H$ \cite{L,R}. In this paper we obtain
results in $\mathbb{H}^{2}\times \mathbb{R}$ in the same direction: given $H\in \left( 0,1/2\right) $, if $\Omega $ is included in a region of $\mathbb{
H}^{2}\times \left\{ 0\right\} $ bounded by two equidistant hypercycles $\ell(H)$ apart, we
show that, if the geodesic curvature of $\partial \Omega $ is bounded from below by $-1,$ then there is an $H$-graph $G$ over $\Omega $ with $\partial G=\partial \Omega$. We also present more refined existence results involving the curvature of $\pO,$ which can also be less than $-1.$

\vspace{.15cm}
\noindent{\it MSC2010:} 
Primary 53A10, Secondary 53C42.

\end{abstract}

\section{Introduction}

Surfaces of constant mean curvature (cmc) in Riemannian manifolds are a classical subject in Differential Geometry. There are many existence results on closed cmc surfaces and on cmc surfaces with boundary (Plateau problem). The Plateau problem becomes a Dirichlet PDE problem if one searches surfaces that are graphs. In this paper we deal with existence results of cmc graphs with boundary included in a plane $\HH\times\{0\}$ in the manifold $\HH\times\R.$

It is quite natural that existence results for graphs with planar boundary can be
obtained with weaker hypotheses than the ones for the Dirichlet problem for
any continuous boundary data. In the Euclidean case, for example, given a
bounded convex $C^{2,\alpha }$ domain $\Omega \subset \mathbb{R}^{2}$,
there is a cmc graph with boundary $\partial \Omega$ if the curvature of $\partial \Omega $ $(k_{\partial \Omega })$ is greater than $H$, while the existence of a
cmc $H$-graph over $\Omega $ taking any continuous boundary data is only
guaranteed by $k _{\partial \Omega }$ bounded from below by $2H$.
Besides, the \textit{a piori} height estimate $|u|<a$ for some $a<1/2H$ also guarantees the existence of cmc graphs with vanishing boundary data in bounded convex $C^{2,\alpha}$ domains (Theorem 3 of \cite{R}). As a consequence of this
result, we draw attention to the following result about unbounded convex
domains:

\begin{theorem}[Theorem 1.2 and 1.4 of \cite{L} or Corollary 3 of \cite{R}]\label{theoEuc}
The\newline Dirichlet problem with zero boundary data associated to the cmc equation can be solved for convex domains included in a strip of width $1/H.$
\end{theorem}

In \cite{L}, it is also proved that if $\Omega$ is an unbounded convex domain and there exists a cmc $H$ graph over $\Omega$ with boundary $\pO,$ then $\Omega$ must be included in a strip of width $1/H.$ We are interested in generalize this results to $\hh\times\R,$ that is: Given $\Omega\subset \hh\times\{0\},$ when is there a cmc $H$ surface with boundary $\partial \Omega?$

The non existence result presented in \cite{L} relies on the compactness of the cmc $H$ sphere in $\R^3$ and on the structure of convex sets in the Euclidean plane. Unfortunately we could not obtain non existence results in $\hh\times \R.$

An interesting point about Theorem \ref{theoEuc} is that the lower bound on $k_{\pO}$ does not depend on $H.$ Such a result also holds in $\HH\times\R,$ see Corollary \ref{C1}, item \ref{corC1-iii}: Suppose $H\in (0,1/2)$ and let $\Omega \subset \mathbb{H}%
^{2} $ be a $C^{2}$ domain with $k_{\pO}\geq -1$ and $\overline{\Omega}$ contained in a region bounded by two hypercycles
which are equidistant $\ell=\ell(H) $ to a fixed geodesic, where $\ell
\left( H\right) $ is given by \eqref{iH}. Then there exists an $H-$graph with boundary $\pO.$

\

Let $\mathbb{H}^{2}$ be the hyperbolic plane, $\Omega
\subset \mathbb{H}^{2}$ a $C^{2,\alpha }$ domain and $H>0$, and consider the Dirichlet problem 
\begin{equation}
\left\{ 
\begin{array}{l}
Q_{H}\left( u\right) :=\mathrm{div}\left( \frac{\nabla u}{\sqrt{1+\left\vert
\nabla u\right\vert ^{2}}}\right) +2H=0 \text{ in } \Omega \\ 
u=0 \text{ on } \partial\Omega \text{ and } u\in C^2(\Omega)\cap C^0(\overline{\Omega}),%
\end{array}%
\right.  \label{DP}
\end{equation}%
%\textcolor{red}{quero tirar isso $u\in C^{2,\alpha }\left( \overline{\Omega}\right)$ pq falamos em sol $C^0.$ Mas nao sei como.} 
where $\nabla $ and $%
\mathrm{div}$ are the gradient and divergent in $\mathbb{H}^{2}$,
respectively. If $u$ is a solution of the Dirichlet problem \eqref{DP} then
the graph of $u$, denoted by $G\left( u\right)$ and oriented with normal vector pointing downwards, is a cmc $H$ surface in $\mathbb{H}^{2}\times \mathbb{R}$ with boundary $%
\partial G\left( u\right) \subset \mathbb{H}^{2}\times \left\{ 0\right\}
\equiv \mathbb{H}^{2}$.

We show that, considering a (possibly unbounded) domain $\Omega $ contained
in a region of $\mathbb{H}^{2}$ bounded by two equidistant hypercycles,
under hypotheses involving the distance between them, their curvatures, the
curvature of $\partial \Omega $ and assuming $H\in \left( 0,1/2\right) $,
problem \eqref{DP} is solvable (Corollaries \ref{C3n} and \ref{C1}).

These results are consequence of a technical theorem (Theorem \ref{MT}),
which was inspired in Theorem 3 of \cite{R} and, essentially states that for 
$H\in \left( 0,1/2\right) $ and $\Omega $ bounded, setting $\kappa
:=\inf_{\partial \Omega }k_{\partial \Omega }$, if any solution $u$ of %
\eqref{DP} satisfies the \textit{a priori} height estimate $\left\vert
u\right\vert \leq a<F\left( \kappa ,H\right) $ (see Remark \ref{rmk_FkH} for an expression for $F$), with an additional
hypothesis on $\Omega $ in the case $\kappa \leq -1$, then \eqref{DP} is
solvable. 
We observe that the case $\kappa >2H$ for bounded domains is already considered in
Theorem 1.5 of \cite{Sp}.

\

There are many height estimates for embedded compact $H$-surfaces $S\subset 
\mathbb{M}^{2}\times \mathbb{R}$, $H>0$, with boundary in a slice $\mathbb{M}%
^{2}\times \left\{ t_{0}\right\} $, given in terms of different hypotheses
such as the area of $S\cap \mathbb{M}^{2}\times \left\{ t\geq t_{0}\right\} $%
, the Gauss map of the immersion (see \cite{LM}, \cite{M} for the Euclidean
case) or the Gauss curvature of $\mathbb{M}^{2}$ (see \cite{HLR}, \cite{AEG}
and, for Hadamard manifolds, see \cite{LR} and generalization for warped
products in \cite{FPS}). Beautiful existence results for $H$-graphs in $%
\mathbb{R}^{3}$ with boundary in a plane were obtained as a consequence of
such estimates, with hypotheses on the length of $\partial \Omega $ or the
area of $\Omega $ (see, for example, Corollary 4 of \cite{LM} and Corollary
5 of \cite{R}).

We point out that, in the Euclidean case, the Dirichlet problem \eqref{DP} for
arbitrary continuous boundary data can be reduced to the zero boundary data (see \cite{R2}) and it is quite natural to expect similar reductions to work for $\mathbb{H}^{2}\times \mathbb{R}$.

\section{Preliminaries}

Given a domain $\Omega \subset \mathbb{R}^{2}$ with $k_{\partial
\Omega }>H$, by using a hemisphere as a barrier one estimates the height of
a solution to the Dirichlet problem \eqref{DP}. Afterwards, moving the
hemisphere slightly downwards and touching it at each point of $\partial
\Omega $, one estimates the gradient at the boundary. These estimates and
the classical PDE theory guarantee the existence of a constant mean
curvature $H$ graph over $\Omega $ with boundary $\partial \Omega $. A
natural way to improve this result would be to ask whether it holds for
unbounded domains, and therefore the hypothesis $k_{\partial \Omega }>H$
cannot be required anymore. Hence a surface that works as barrier must have
its boundary in $\mathbb{R}^{2}$ with curvature smaller than $H$. The
simplest non spherical cmc surface is the cylinder, which if we consider
having a horizontal axis is a bi-graph over a strip of width $1/H$. In \cite%
{R} and in \cite{L}, it is used as a barrier to prove the result stated in
the abstract.

Our main goal in this manuscript is to extend the result above to $\mathbb{H}%
^{2}\times \mathbb{R}$. Nevertheless $\mathbb{H}^{2}\times \mathbb{R}$ is
less symmetric than $\mathbb{R}^{3}$ and so are the surfaces that play the
role of the half cylinders. However, it is possible to find surfaces that
work as barriers in our context, which is done in the following lemmas. None of the surfaces described here is new in the literature, see for instance \cite{BRWE} and \cite{NR} among many others. We recall them here in order to fix some notation and describe some important properties. The authors believe that the least known property is an estimate on the size of the set where the barriers are positive (inequalities \eqref{eq-deltaineq} and \eqref{eq-Deltaineq} below), which was obtained in \cite{KST}.

\begin{lemma}
\label{dgeo}Let $H\in \left( 0,1/2\right) $ and $l>0$ be given. Let $\alpha
_{1}$ and $\alpha _{2}$ be two hypercycles in $\mathbb{H}^{2}$, equidistant $%
l$ of a given geodesic of $\mathbb{H}^{2}$. There is an $H$-graph $\mathfrak{%
G}=\mathfrak{G}\left( l,H\right) $ over the connected region in $\mathbb{H}%
^{2}$ bounded by $\alpha _{1}\cup \alpha _{2}$, contained in $\mathbb{H}%
^{2}\times \left\{ t\geq 0\right\} $, with $\partial \mathfrak{G}=\alpha
_{1}\cup \alpha _{2}$ and whose height is 
\begin{equation}
h_H(l):=\frac{2H}{\sqrt{1-4H^2}}\ln\left[\left(\frac{\sqrt{1-4H^2}+\sqrt{1-4H^2\tanh^2(l)}}{\sqrt{1-4H^2}+1}\right)\cosh(l)\right] \label{hG}
\end{equation}

\end{lemma}

\begin{proof}
Let $\gamma $ be the given geodesic. Set $d\left( z\right) =d_{\mathbb{H}%
^{2}}\left( z,\gamma \right) $, $z\in \mathbb{H}^{2}$. Let $\psi \in
C^{2}\left( [0,\infty \right) )$ be given by 
\begin{equation}
\psi \left( d\right) =2H\int_{d}^{l}\frac{\tanh s}{\sqrt{1-4H^{2}\tanh ^{2}s}%
}ds,  \label{psi2}
\end{equation}%
which is well-defined for $H\in \left( 0,1/2\right) $. Then observe that for
a function of the form $u=\tilde{u}\circ d$ the PDE in \eqref{DP} rewrites
as 
\begin{equation*}
\left(\frac{\tilde{u}^{\prime }\left( d\right) }{\sqrt{1+\left[ 
\tilde{u}^{\prime }\left( d\right) \right] ^{2}}}\right)^{\prime }+\left(\frac{\tilde{u}'(d) }{\sqrt{1+\left[ 
\tilde{u}^{\prime }(d) \right] ^{2}}}\right) \tanh d +2H=0.
\end{equation*}
Therefore $u(z)=\psi (d(z))$ satisfies \eqref{DP} in the domain
bounded by $\alpha _{1}\cup \alpha _{2}$ and it is non negative.
\end{proof}

\begin{definition}\label{def-util}
Given $H\in \left( 0,1/2\right) $, $r>\tanh ^{-1}\left( -2H\right) $ and $%
\rho >0$, we define%
\begin{eqnarray}
c_{H}\left( r,t\right) &=&\frac{\cosh r+2H\left( \sinh r-\sinh (r+t)\right) 
}{\cosh (r+t)}%
\text{, }t\geq 0,  \label{fHrt} \\
s_{H}\left( \rho ,t\right) &=&\frac{\sinh \rho +2H\left( \cosh \rho -\cosh
\left( \rho +t\right) \right) }{\sinh \left( \rho +t\right)}\text{, }t \geq 0,  \label{gHrt} \\
z_{H}(r) &=&\sinh ^{-1}\left( \sinh r+\frac{\cosh r}{2H}\right) -r,
\label{MHr} \\
Z_{H}(\rho ) &=&\cosh ^{-1}\left( \cosh \rho +\frac{\sinh \rho }{2H}\right)
-\rho  \label{NHr}\\
a_{H}(r)&=&\int_{0}^{z_{H}(r)}\frac{c_{H}\left( r,t\right)}{\sqrt{1-[c_H(r,t)]^2}} dt, \label{hHS}\\ 
A_{H}(\rho)&=&\int_{0}^{Z_{H}(\rho)}\frac{s_{H}\left( \rho,t\right)}{\sqrt{1-[s_H(\rho,t)]^2}} dt\label{hNod}
\end{eqnarray}%
and we denote by $\delta _{H}\left( r\right) $ and $\Delta _{H}\left( \rho
\right) $ the positive numbers given by 
\begin{equation}
\int_{0}^{\delta _{H}(r)}\frac{c_{H}\left( r,t\right)}{\sqrt{1-[c_H(r,t)]^2}} dt=0\text{ and }%
\int_{0}^{\Delta _{H}\left( \rho \right) }\frac{s_{H}\left( \rho,t\right)}{\sqrt{1-[s_H(\rho,t)]^2}} dt=0.
\label{dHr}
\end{equation}
\end{definition}

\begin{remark}
The functions $c_H(r, \cdot)$ and $s_H(\rho, \cdot)$ are well defined in $[0,\infty),$ and satisfy $|c_H(r, t)|<1,$ $|s_H(r, t)|<1$ for $t>0.$ Both are decreasing functions which assume the value zero at $z_H$ and $Z_H,$ respectively. Despite taking the value $1$ at $t=0$, the integrals in \eqref{dHr} are well defined. Besides, we may define
\begin{equation}
a_H\left(\tanh ^{-1}(-2H)\right):=\underset{r\rightarrow \tanh ^{-1}\left( -2H\right) }{\lim }
\int_{0}^{z_{H}(r)}\frac{c_{H}\left( r,t\right)}{\sqrt{1-[c_H(r,t)]^2}} dt=+\infty.  \label{r2H}
\end{equation}
\end{remark}

\begin{lemma}
\label{L2mhr}Let $H\in (0,1/2)$ and $r>\tanh ^{-1}\left( -2H\right) $ be
given. Let $\alpha $ be a hypercycle in $\mathbb{H}^{2}$ with geodesic
curvature $k_{\alpha }=-\tanh r$. There are a hypercycle $\beta $ in $%
\mathbb{H}^{2}$, equidistant to $\alpha $, which satisfies 
\begin{equation}\label{eq-deltaineq}
d\left( \alpha ,\beta \right) =\delta _{H}(r)\geq 2z_H\left( r\right)
\end{equation}
and an $H$-graph $\mathcal{G}=\mathcal{G}\left( r,H\right) $ over the
connected component of $\mathbb{H}^{2}$ bounded by $\alpha \cup \beta $,
contained in $\mathbb{H}^{2}\times \left\{ t\geq 0\right\} $ and such that $%
\partial \mathcal{G}=\alpha \cup \beta $. The height of $\mathcal{G}$ is $a_{H}(r).$
\end{lemma}

\begin{proof}
Once again we use the distance function to turn the Dirichlet problem %
\eqref{DP} into an ODE problem.

Let $A$ be the connected component of $\mathbb{H}^{2}\backslash \alpha $
such that $k_{\partial A}=-\tanh r$ with the inner orientation. Set $d\left(
z\right) =d_{\mathbb{H}^{2}}\left( z,\alpha \right) $, $z\in A$. For $\gamma $
being the geodesic equidistant to $\alpha $, if the distance between $\gamma$
and $\alpha $ is $|r|$ and $\gamma \subset A$, observe that $r<0$. Otherwise,
if $\gamma \subset \mathbb{H}^{2}\backslash A$, then $r>0$.

Take $\tilde{u}\in C^{2}\left( (0,\infty \right) )\cap C^{0}\left( [0,\infty \right)
)$ satisfying $\tilde{u}\left( 0\right) =0$, $\tilde{u}^{\prime }\left( d\right) \rightarrow
+\infty $ when $d\rightarrow 0$ (see \cite{KST} for details), and such that
the graph of $u=\tilde{u}\circ d\in C^{2}\left( A\right) \cap C^{0}\left( \overline{A%
}\right) $ is an $H$-graph. Then 
\begin{equation}
\tilde{u}(d)=\int_{0}^{d}\frac{c_{H}\left( r,t\right)}{\sqrt{1-[c_H(r,t)]^2}} dt,  \label{bar2}
\end{equation}%
and clearly $u|_{\alpha }=0$. It is immediate to see that $z_H\left(
r\right) $ is the maximum point of the function $\tilde{u}$ and $\tilde{u}$ goes to $-\infty 
$ when $d\rightarrow +\infty $. Set 
\begin{equation*}
\beta :=\left(G\left(u\right) \backslash \alpha \right)\cap \left(\mathbb{H}%
^{2}\times \left\{ 0\right\} \right).
\end{equation*}%
Notice that $\beta \subset A$ is a hypercycle which satisfies $d\left(
\alpha ,\beta \right) =\delta _{H}\left( r\right) $, where $\delta
_{H}\left( r\right) $ is given by \eqref{dHr}.

Now, we prove that $\tilde{u}$ is non-negative in $\left[ 0,2z_H\left( r\right) %
\right] $, that is $2z_H\left( r\right) \leq \delta _{H}\left( r\right) $.
Notice that is enough to show that 
\begin{equation*}
\left\vert \tilde{u}^{\prime }\left( z_H\left( r\right) +s\right) \right\vert \leq
\left\vert \tilde{u}^{\prime }\left( z_H\left( r\right) -s\right) \right\vert ,%
\text{ }s\in \left( 0,z_H\left( r\right) \right) ,
\end{equation*}%
which is equivalent to $-\tilde{u}^{\prime }\left( z_H\left( r\right) +s\right)
\leq \tilde{u}^{\prime }\left( z_H\left( r\right) -s\right) .$

Since 
\begin{equation*}
\tilde{u}^{\prime }\left( s\right) =\frac{c_H \left( s\right) }{\sqrt{1-\left[ c_H
\left( s\right) \right] ^{2}}}.
\end{equation*}%
and $x\mapsto x\left( 1-x^{2}\right) ^{-1/2}$ is increasing
in the interval $(0,1)$, it is enough to show that 
\begin{equation}
-c_H \left( z_H\left( r\right) +s\right) \leq c_H \left( z_H\left(
r\right) -s\right) .  \label{d1_2MHr}
\end{equation}%
We have $0=\tilde{u}^{\prime }\left( z_H\left( r\right) \right) =c_H \left(
z_H\left( r\right) \right) $ and, then, we can rewrite $c_H \left(
t\right) $ as 
\begin{equation}\label{eq-conta}
c_H \left( t\right) =\frac{2H\left[ \sinh \left( r+z_H\left( r\right)
\right) -\sinh \left( r+t\right) \right] }{\cosh \left( r+t\right) }.
\end{equation}%
Plugging \eqref{eq-conta} in \eqref{d1_2MHr}, expanding $\cosh $ and $\sinh $ of sums and observing that%
\begin{equation*}
\sinh \left( r+z_H\left( r\right) \right) =\sinh r+\frac{\cosh r}{2H},
\end{equation*}%
a straightforward computation gives us that \eqref{d1_2MHr} holds and the
result follows.
\end{proof}

As we will see in the next section, the graphs given by the lemmas above
will provide our barriers relatively to the Dirichlet problem \eqref{DP} in
the case $-1<\inf_{\partial \Omega }k_{\partial \Omega }\leq 2H$. Relatively to the case $\inf_{\partial \Omega
}k_{\partial \Omega }\leq -1$, we also have barriers. They are pieces of $H$%
-nodoids, described in the following lemma:

\begin{lemma}
\label{L3nhr}Let $H\in \left( 0,1/2\right) $ and $\rho >0$ be given. Let $%
C_{\rho }\subset \mathbb{H}^{2}$ be a (hyperbolic) circle of radius $\rho $.
There is a compact $H$-graph $\mathbb{G=G}\left( \rho ,H\right) $ contained
in $\mathbb{H}^{2}\times \left\{ t\geq 0\right\} $, over an annulus in $%
\mathbb{H}^{2}$ whose boundary is the union of the concentric circles $%
C_{\rho }$ and $C_{\rho ^{\ast }}$, with 
\begin{equation}\label{eq-Deltaineq}
\rho ^{\ast }:=\rho +\Delta _{H}\left( \rho \right) \geq \rho +2Z_H\left(
\rho \right).
\end{equation}
Moreover, the height of $%
\mathbb{G}$ is $A_H(\rho).$ The quantities
$\Delta _{H}\left( \rho \right),$ $Z_H\left(
\rho \right)$ and $A_H(\rho)$ were presented in Definition \ref{def-util}.
\end{lemma}

\begin{proof}
Let $B_{\rho }\subset \mathbb{H}^{2}$ be the open disk of radius $\rho $
such that $\partial B_{\rho }=C_{\rho }$ and set $A=\mathbb{H}^{2}\backslash 
\overline{B}_{\rho }$ and $d\left( z\right) =d_{\mathbb{H}^{2}}\left(
z,C_{\rho }\right) $, $z\in A$. The function $u =\tilde{u} \circ d\in
C^{2}\left( A\right) \cap C^{0}\left( \overline{A}\right) $ given by%
\begin{equation}
\tilde{u} \left( d\right) =\int_{0}^{d}\frac{s_{H}\left( \rho,t\right)}{\sqrt{1-[s_H(\rho,t)]^2}} dt  \label{bar3}
\end{equation}%
satisfies $Q_{H}\left( u \right) =0$ in $A$, with $u |_{C_{\rho
}}=0$. Moreover, $\tilde{u} $ has its maximum point at $Z_H\left( \rho \right) $
and is non negative in $\left[ 0,2Z_H\left( \rho \right) \right] $. The
proof of these facts follows the same steps of the lemma above, just noting
that, now, $\Delta d=\coth \left( \rho +d\right) $ and the conditions on the
function $\tilde{u} $ are $\tilde{u} \left( 0\right) =0$ and $\tilde{u} ^{\prime }\left(
d\right) \rightarrow +\infty $ when $d\rightarrow 0$ (see \cite{KST} for
details).
\end{proof}

The next result presents a relation between the barriers constructed in the
previous lemmas.

\begin{proposition}
\label{P_conjs} Let $H\in \left( 0,1/2\right) $ be given. The function $a_H$ is decreasing, the function $A_H$ is increasing and both have the same limit at infinity, which is 
\begin{equation}
a_H(\infty) =\frac{\pi }{2}-\frac{4H}{\sqrt{1-4H^{2}}}\tanh ^{-1}\left( \frac{%
1-2H}{\sqrt{1-4H^{2}}}\right). \label{hNinf}
\end{equation}
Besides, if $\tanh ^{-1}\left( -2H\right) \leq r<0$, then $h_{H}(\left\vert r\right\vert)<a_H(r) $ and $%
\left\vert r\right\vert <z_H\left( r\right) $, where $h_{H}$ is given by \eqref{hG} and the other quantities were presented in Definition \ref{def-util}.
\end{proposition}

\begin{proof}
First notice that straightforward
computations give us that $Z_H$ and $z_H$ are increasing
and decreasing functions, respectively, with the same limit at infinity 
\begin{equation}
z_H(\infty) =\log
\left( \frac{1}{2H}+1\right) \label{NMcomp}
\end{equation}%
Also, we have $\frac{\partial s_{H}}{\partial\rho}\left( \rho ,t\right)>0, $ $\frac{\partial c_{H}}{\partial r}\left( r,t\right)<0 $ and 
\begin{equation}
\underset{\rho \rightarrow +\infty }{\lim }\frac{s_{H}\left( \rho ,t\right)}{\sqrt{1-[s_{H}\left( \rho ,t\right)]^2}} =\frac{%
-2H+(1+2H)e^{-t}}{\sqrt{1-(-2H+(1+2H)e^{-t})^2}}  \label{gfcomp}
\end{equation}
which coincides with the limit as $r$ goes to infinity of $$\frac{c_{H}\left( r ,t\right)}{\sqrt{1-[c_{H}\left( r,t\right)]^2}} .$$
Then, both functions, $a_H$ and $A_H$ converge to
\begin{equation*}
\displaystyle{\int_0^{\log \left( \frac{1}{2H}+1\right) }\frac{%
-2H+(1+2H)e^{-t}}{\sqrt{1-(-2H+(1+2H)e^{-t})^2}}dt}
\end{equation*}
at infinity. This integral results the expression presented in the statement.  Moreover, for all $\rho >0$ and $r> \tanh ^{-1}\left( -2H\right) $, we have 
that $A_H(\rho) \leq a_H\left(
r\right).$

For the comparison with $h_{H}$, take two hypercycles $\alpha
_{1} $ and $\alpha _{2}$ equidistant $\left\vert r\right\vert $ to a same
geodesic $\gamma $. Let $A$ be the connected component of $\mathbb{H}%
^{2}\backslash \alpha _{1}$ which contains $\gamma $. Then the distance, with
sign, from $\alpha _{1}$ to $\gamma $ is $r<0$. From Lemma \ref{L2mhr}, there
is a cmc $H-$graph $\mathcal{G}\subset \mathbb{H}^{2}\times \left\{ t\geq
0\right\} $ with boundary $\alpha _{1}\cup \beta $ in $\mathbb{H}^{2}$,
where $\beta $ is a hypercycle which satisfies $d\left( \alpha _{1},\beta
\right) \geq 2z_H\left( r\right) $ and, moreover, $\mathcal{G}$ is
vertical at $\alpha _{1}$. On the other hand, by Lemma \ref{dgeo}, there is
a cmc $H$-graph $\mathfrak{G\subset }\mathbb{H}^{2}\times \left\{ t\geq
0\right\} $ with boundary $\alpha _{1}\cup \alpha _{2}$ which is not
vertical at $\alpha _{1}$ (neither at $\alpha _{2}$). Now the tangent
principle implies that $h_{H}(\left\vert r\right\vert)<a_H(r) .$

The inequality $|r|<z_H(r)$ follows from the definition of $z_H(r).$
\end{proof}

\section{Main results}

Let us prove Theorem \ref{MT} which shows that an a priori height estimate
is enough for the existence result in \eqref{DP}. After that we obtain the
existence results.

\begin{definition}
We say that a $C^{2}$ domain $\Omega \subset \mathbb{H}^{2}$ satisfies
the exterior circle condition of radius $\rho >0$ if, for each $p\in
\partial \Omega $, there is hyperbolic circle $C_{\rho }\subset \mathbb{H}%
^{2}\backslash \overline{\Omega }$ of radius $\rho $ which is tangent to $%
\partial \Omega $ at $p$.
\end{definition}

We observe that if $-1\leq \inf_{\partial \Omega }k_{\partial \Omega }$,
then $\Omega $ satisfies the exterior circle condition of radius $\rho $ for
all $\rho >0$.

\begin{theorem}
\label{MT}Let $\Omega \subset \mathbb{H}^{2}$ be a bounded $C^{2,\alpha }$
domain and let $H\in (0,1/2)$ be given. Set $\kappa =\inf_{\partial \Omega
}k_{\partial \Omega }.$ 
\begin{enumerate}[label=(\roman{enumi})]
\item\label{teo32-i} If $\kappa >2H$ then the Dirichlet problem \eqref{DP} has a solution in $%
C^{2,\alpha }(\overline{\Omega }).$
\item\label{teo32-ii} If $\kappa \in (-1,2H]$ and there exists $0<a<a_H\left(
\tanh ^{-1}\left( -\kappa \right)\right) $ such that any solution $u$ of (%
\ref{DP}) satisfies the a priori height estimate $\sup_{\Omega }|u|\leq a$,
where $a_H$ is given by \eqref{hHS}, then there is a solution $%
u\in C^{2,\alpha }(\overline{\Omega })$ to the Dirichlet problem \eqref{DP}. 
\item\label{teo32-iii} If $\kappa <-1$, $\Omega $ satisfies the exterior circle condition of
radius $\coth ^{-1}\left( -\kappa \right) $ and there exists $0<b<A_H\left( \coth ^{-1}\left( -\kappa \right) \right) $ such that any
solution $u$ to \eqref{DP} satisfies the a priori height estimate $%
\sup_{\Omega }|u|\leq b$, where $A_H$ is given by \eqref{hNod},
then the Dirichlet problem \eqref{DP} has a solution $u\in C^{2,\alpha }(%
\overline{\Omega }).$
\item\label{teo32-iv} If $\kappa =-1$ and there exists $0<b<a_H(\infty)$ such that any solution $u$ to \eqref{DP} satisfies the a priori
height estimate $\sup_{\Omega }|u|\leq b$, where $a_H(\infty) $ is given by \eqref{hNinf}, then the Dirichlet problem \eqref{DP} has a solution $u\in C^{2,\alpha }(\overline{\Omega })$.
\end{enumerate}
\end{theorem}

\begin{remark}\label{rmk_FkH}
As a consequence of the above result, the function $F(\kappa,H)$ mentioned in the Introduction can be computed as 
$$F(\kappa,H)=\left\{\begin{array}{l}
+\infty \text{ if } \kappa\geq 2H\\
a_H(\tanh^{-1}(-\kappa)) \text{ if } \kappa \in [-1,2H)\\
A_H(\coth^{-1}(-\kappa)) \text{ if } \kappa <-1. \end{array}\right.
$$
\end{remark}

\begin{proof}
The first item is a well-known result which holds also for continuous
boundary data, see Theorem 1.5 in  \cite{Sp}.

In order to prove the other cases we use barriers to obtain a priori
estimates to $\sup_{\partial \Omega }|\nabla u|$. Then, by standard elliptic
PDE theory (see \cite{GT}), the existence result holds.

Our barriers are inspired in the quarters of cylinders in $\mathbb{R}^{3}$
which were used in this same way in \cite{L} and \cite{R}.

\textbf{Proof of \ref{teo32-ii}: }
We first consider the case $\kappa \in \left( -1,2H\right).$

Here, the role of the cylinders will be played by
parts of $\mathcal{G}=\mathcal{G}\left( r,H\right) $, $r>\tanh ^{-1}\left(
-2H\right) $, where $\mathcal{G}$ is described in Lemma \ref{L2mhr}.
Nevertheless, since the sets $\left(\mathbb{H}^{2}\times \left\{ t\right\} \right)\cap 
\mathcal{G}$, $t\geq 0$, are hypercycles, which do not all have the same
geodesic curvature, moving them downwards requires some care. We start this
proof by describing this movement.

Let $\alpha $ be a curve in $\mathbb{H}^{2}$ oriented with normal vector $%
\eta $ and of constant geodesic curvature $\kappa $. Let $r=\mathrm{\tanh }%
^{-1}(-\kappa )$. Let $A$ be the connected component of $\mathbb{H}%
^{2}\backslash \alpha $ for which $\eta $ points and set $d\left( z\right)
=d_{\mathbb{H}^{2}}\left( z,\alpha \right) $, $z\in A.$
Let $\Lambda\subset A$ be the strip of width $z_H\left( r\right) $ bounded by $\alpha\cup \alpha^\star,$ being $z_H\left( r\right) $ given by \eqref{MHr}. Then the graph of $u=\left( \tilde{u}\circ d\right) |_\Lambda$, $\tilde{u}$ given in \eqref{bar2}, has cmc $H.$ Besides $u$ vanishes on $\alpha$ and, since  
$\tilde{u}(z_H(r))=a_H(r) >a,$ $u$ is greater than $a$ on $\alpha^\star.$

Let $r_{1}>0$ be such that $\tilde{u}(r_{1})+a<a_H(r) $.
Define $w_1:\Lambda \rightarrow \mathbb{R}$ by $w_1
(z)=\left( \tilde{u}\circ d\right) \left( z\right) -\tilde{u}(r_{1})$ and let $w$ be the
restriction of $w_1$ to the subset of $\Lambda $ given by $\left\{
z\in \Lambda ;w_1\left( z\right) \geq 0\right\} $. Since $u$ is an
increasing function depending on the distance to $\alpha $, the domain of $w$
is bounded by two curves equidistant to $\alpha $: $\delta $ of distance $
r_{1}$ (geodesic curvature $-\tanh (r+r_{1})=\kappa _{1}$) on which $w$
vanishes and $\alpha^\star$ of distance $z_H(r)$ on which $w$ is greater than $
a $. The graphs of functions $w$ well located in $\mathbb{H}^{2}$ will work
as our barriers. Since 
\begin{equation*}
k_{\partial \Omega }\geq \kappa =-\tanh \left( r\right) >-\tanh \left(
r+r_{1}\right) =\kappa _{1},
\end{equation*}%
for each $p\in \partial \Omega $ there is a curve $\delta _{p}$ tangent to $%
\partial \Omega $ at $p$, contained in $\Omega ^{C}$ and of constant
geodesic curvature $\kappa _{1}$ (oriented by $\eta _{\delta }$ such that $%
\eta _{\delta }(p)=\eta _{\partial \Omega }(p)$). Let $w_{p}: \Lambda_p\to\R$ be the barrier
described above that depends on the distance to $\delta _{p}$. 

We claim that
if $u\in C^{2,\alpha }(\overline{\Omega })$ is a solution to \eqref{DP},
then $u\leq w_{p}$ in the domain $\Omega _{p}$ where both functions are
defined.

Notice that  $\partial \Omega _{p}=\Gamma _{1}\cup \Gamma
_{2}$, where $\Gamma _{1}=\partial \Omega_p \cap \partial\Omega$ and $%
\Gamma _{2}=\partial \Omega _{p}\cap\partial \Lambda_p$. We have 
$u|_{\Gamma _{1}}=0\leq w_{p}|_{\Gamma _{1}}$
and, since $\sup_{\Omega }|u|<a$ we have $u|_{\Gamma _{2}}<a<w_{p}|_{\Gamma
_{2}}$. Then $u|_{\partial \Omega _{p}}\leq w_{p}|_{\partial \Omega _{p}}$. Now, since $u\left(p\right) =w_{p}\left( p\right) $ and $\sup_{\overline{\Omega }%
_{p}}\left\vert \nabla w_{p}\right\vert \leq \tilde{u}(r_{1})<\infty $, we
conclude that $w_{p}$ is supersolution relatively to the domain $\Omega _{p}.$ So, standard elliptic PDE theory guarantees the
existence of a solution $u\in C^{2,\alpha }(\overline{\Omega })$ to the
Dirichlet problem \eqref{DP}.

Relatively to the case $\kappa =2H$, assuming an a priori height estimate $a,$ we see from \eqref{r2H} that there is $r>\tanh ^{-1}\left( -2H\right) $ such that $0<a<a_H(r) $. For such $r$, as $-\tanh r\leq 2H=\kappa $, we can proceed as above and this concludes the proof of item \ref{teo32-ii}.

\textbf{Proof of \ref{teo32-iii}: } Set $\rho _{0}=\coth ^{-1}\left( -\kappa \right) $.
Since $\Omega $ satisfies the exterior circle condition of radius $\rho _{0}$%
, given $p\in \partial \Omega $ there is $q\in \mathbb{H}^{2}\backslash \overline{%
\Omega }$ and a circle $C_{\rho _{0}}\left( q\right) \subset \mathbb{H}^{2}\backslash 
\overline{\Omega }$ such that $C_{\rho _{0}}\left( q\right) $ is tangent to $%
\partial \Omega $ at $p$.

Since $\rho \mapsto A_H(\rho)$ is a continuous increasing function, there are $\rho_1<\rho_0$ and $s<0$ satisfying $b<A_H(\rho_1) +s$ and $\rho_1+\varepsilon<\rho_0.$ The positive number $\varepsilon=\varepsilon(H,\rho_1)$ is given by $\varepsilon=\tilde{u}^{-1}(-s)$ for $\tilde{u}:[0,Z_H(\rho_1)] \to \R$ defined in \eqref{bar3}:
$$\tilde{u} (d) =\int_{0}^{d}\frac{s_{H}\left( \rho_1,t\right)}{\sqrt{1-[s_H(\rho_1,t)]^2}} dt.$$ 

By Lemma \ref{L3nhr}, setting $d\left( z\right) =d_{\mathbb{H}^{2}}\left( z,C_{\rho_1
}\left( q\right) \right) $,  the graph $\mathbb{G}$ of the function $u \left( z\right) =\tilde{u} \circ d\left(
z\right) +s$ defined on the annulus 
\begin{equation*}
A:=\left\{ z\in \left[ B_{\rho_1 }\left( q\right) \right] ^{C},0\leq d\left(
z\right) \leq Z_H\left( \rho_1 \right) \right\} ,
\end{equation*}%
has cmc $H$ and is contained in $%
\mathbb{H}^{2}\times \left\{ t\geq s\right\}.$ Besides 
$\mathbb{G}\cap \left(\mathbb{H}^{2}\times \left\{ 0\right\}\right)$
is a circle $C_{\rho_1+\varepsilon}\left( q\right) $ of radius $\rho_1+\varepsilon<\rho_0.$ The connected component of $\partial \mathbb{G}$ which is contained in $\mathbb{H}^{2}\times\left\{ t>0\right\} $ is 
$\mathbb{G}\cap \left( \mathbb{H}^{2}\times \left\{
A_H(\rho_1) +s\right\} \right)\subset  \mathbb{H}^{2}\times \{t>b\} .$

Let 
$\Lambda=\left\{ z\in A;\rho_1+\varepsilon \leq d\left( z\right) \leq
Z_H\left( \rho_1 \right) \right\}$ be the annulus where $u$ is positive
and set $w=u |_{\Lambda}$. We have $\sup_{ \Lambda
}\left\vert \nabla w\right\vert=\tilde{u}'(\varepsilon) <\infty $ and
\begin{equation*}
b<\sup_{\Lambda}w=A_H(\rho_1) +s.
\end{equation*}
Now, consider
the geodesic radius $\gamma _{p}$ linking $q$ to $p$, with $\gamma
_{p}\left( 0\right) =q$ and $\gamma _{p}\left( \rho _{0}\right) =p$ and
translate the graph of $w$ along $\gamma _{p}$ until its lower boundary
touches $p.$ We name $w_p$ the function that has this translated surface 
as graph and $\Lambda_p$ be the domain of $w_p.$ 

Set $\Omega_p=\Lambda_p\cap \Omega $. Note that $\partial \Omega_p=\Gamma _{1}\cup \Gamma _{2}$ where 
$\Gamma _{1}=\partial \Omega_p\cap \partial \Omega$ and $\Gamma _{2}=\partial
\Omega_p\cap \partial \Lambda_p.$
If $u$ is a solution of \eqref{DP}, then $u|_{\Gamma _{1}}=0\leq w_p|_{\Gamma
_{1}}$. On the other hand, by hypothesis, $u|_{\Gamma _{2}}<b<w_p|_{\Gamma _{2}}.$ It follows that $w_p$ is a supersolution
relatively to the domain $\Omega_p$, that is, $u\leq w_p$ in $\Omega_p$. The result follows, now, from the fact that $\sup_{\Omega_p}\left\vert \nabla w_p\right\vert <\infty $.

\textbf{Proof of \ref{teo32-iv}: } Since $\kappa =-1$, $\Omega $ satisfies the exterior
circle condition for all $\rho >0$. From Proposition \ref{P_conjs} it
follows that there is $\rho $ large enough such that $b<A_H\left(
\rho\right) $. Now, just proceed as in the case \ref{teo32-iii}.
\end{proof}

We observe that in Theorem 3 of \cite{R}, the height estimate does not
depend on $k_{\partial \Omega }$, it only requires $\Omega $ to be convex.
Here, instead of assuming $\Omega $ convex, we require $k_{\partial \Omega
}\geq -1$ and an analogous result holds.

\begin{corollary}
\label{C0}Let $\Omega \subset \mathbb{H}^{2}$ be a bounded $C^{2,\alpha }$
domain and let $H\in (0,1/2)$ be given. If $k_{\partial \Omega }\geq -1$ and
any solution $u$ to \eqref{DP} satisfies the a priori height estimate $
\sup_{\Omega }|u|<a_H(\infty)$, where $a_H(\infty)$ is given by \eqref{hNinf}, then the
Dirichlet problem \eqref{DP} has a solution $u\in C^{2,\alpha }(\overline{
\Omega }).$
\end{corollary}

\begin{proof}
From Proposition \ref{P_conjs}, $a_H(r)$ given by
\eqref{hHS} is decreasing with $r$ and converges to $a_H(\infty)$ as $r \to \infty.$
Set $r=\tanh ^{-1}\left( -\kappa \right) $ if $\kappa >-1$. Then $%
\sup_{\Omega }|u|<a_H(r) $ and the result follows
from item \ref{teo32-ii} of Theorem \ref{MT}. If $\kappa =-1$, the result is item
\ref{teo32-iv} of Theorem \ref{MT}.
\end{proof}

\begin{corollary}
\label{C3n} Let $H\in (0,1/2)$ and $\Omega \subset \mathbb{H}^{2}$ a $
C^{2} $ domain be given. Set $\kappa =\inf_{\partial \Omega }k_{\partial
\Omega }$. If $\kappa \in (-1,2H)$, let $\delta_H\left( r\right) $ be as defined in \eqref{dHr}, where $r=\tanh ^{-1}\left( -\kappa \right) $. It follows that if 
$\overline{\Omega}$ is contained in a region bounded by two hypercycles $\delta_H(r)$ far apart, one of them of geodesic curvature $\kappa $ when
oriented with the normal vector pointing to $\Omega $, then the Dirichlet
problem \eqref{DP} has a solution in $C^{2}\left( \Omega \right) \cap
C^{0}\left( \overline{\Omega }\right).$
\end{corollary}

Since an explicit expression is always nicer, we remark that Lemma \ref{L2mhr} shows that $2z_H(r)<\delta_H(r)$ and therefore if $\overline{\Omega}$ is contained in a region bounded by two hypercycles $$2\left[\sinh ^{-1}\left( \sinh r+\frac{\cosh r}{2H}\right) -r\right]$$ far apart, one of them of geodesic curvature $\kappa=\tanh(-r) $ when oriented with the normal vector pointing to $\Omega,$ then the existence result also holds.

\begin{proof}
Let $\alpha_1$ be the hypercycle equidistant to a geodesic $\zeta $ mentioned in the statement. It follows
that the distance, with sign, from $\zeta $ to $\alpha_1 $, is $r$. Let $d$ be
the distance function to $\alpha_1 $ in the connected component $A$ of $\mathbb{H}%
^{2}\backslash \alpha_1 $ which contains $\Omega $ and let $u_1=\tilde{u}\circ d$ be as defined in Lemma \ref{L2mhr}, \eqref{bar2}. In Lemma \ref{L2mhr}, we proved that for $\beta_1\subset A$ the curve equidistant $\delta_H(r)>2z_H(r)$ from $\alpha_1,$ it holds that $\beta_1=\left( G\left( u_1\right) \backslash
\alpha_1 \right) \cap \left(\mathbb{H}^{2}\times\{0\}\right)$ and $u_1\geq 0$ in $\Lambda_1,$ for $\Lambda_1$ the connected set bounded by $\alpha_1 \cup \beta_1.$ 
From hypothesis, $\overline{\Omega}\subset \Lambda_1,$ $u_1$ vanishes on $\partial\Lambda_1$ and depends on the distance to the boundary. Therefore there is $\varepsilon>0,$ such that $u_1-\varepsilon>0$ in $\Omega.$ We denote by $u$ the function $u_1-\varepsilon$ and by $\alpha$ and $\beta$ the curves that bound the set $\{u>0\},$ which are equidistant to $\zeta$ and satisfy: $\alpha$ is in the region bounded by $\alpha_1\cup \beta$ and $\beta$ is in the region bounded by $\alpha \cup \beta_1.$ We observe that $\sup_{\Lambda} u=\sup_{\Lambda_1} \left( u_1- \varepsilon\right)=a_H(r)-\varepsilon.$

Fix a point $p_{0}=\zeta(0)\in \zeta $. Given $k\in \mathbb{N}$, let $\Gamma _{\pm k}$ be geodesics orthogonal to $\zeta $, intercepting $
\zeta $ at $\zeta(\pm k).$ Notice that such geodesics are also
orthogonal to $\alpha $ and $\beta$. For $k\in \mathbb{N},$ set $p_{\alpha ,\pm k}:=\alpha \cap \Gamma _{\pm k}$ and $\alpha_k$ the part of $\alpha$ from $p_{\alpha ,-k}$ to $p_{\alpha ,k}.$ Analogously we define $p_{\beta ,\pm k}\in \beta \cap \Gamma _{\pm k}$ and $\beta_k.$ Let $R_k$ be the `rectangle' with these four vertices.
Consider the hyperbolic circles of diameter the segment $\overline{p_{\alpha ,k}p_{\beta
,k}}$, which are tangent to $\alpha $ and $\beta $ at $p_{\alpha ,k}$
and $p_{\beta ,k}$ and also the ones of diameter $\overline{p_{\alpha ,-k}p_{\beta
,-k}}$. Let $\lambda _{k}$ be the semicircles of
such circles with extremes $p_{\alpha ,k}$ and $p_{\beta ,k},$ such that $\lambda_k\subset \left({R_k}\right)^C$ and analogously define $\lambda_{-k}$. They have curvature
greater than $1$ when oriented with the normal pointing to $R_k.$ Now consider the bounded $C^{1}$ domain $\Lambda _{k}$ whose
boundary is the curve%
\begin{equation*}
\partial \Lambda_k=\lambda _{k}\cup \alpha_k \cup \lambda
_{-k}\cup \beta_k .
\end{equation*}%
Given $p\in \partial \Lambda_k$, if $p\in \alpha \cup \beta $, define $w_{p}:\overline{%
\Lambda }_{k}\longrightarrow \mathbb{R}$ by $w_{p}\left( z\right) =u\left(
z\right) $, where $u$ is the function described in the beginning of this proof. If $p\in
\lambda _{k}\cup \lambda _{-k}$, since $k_{\alpha_1}<1 <\min \left\{
k_{\lambda _{k}},k_{\lambda _{-k}}\right\} $, we can take a curve $%
\alpha _{p}$ tangent to $\partial \Lambda_k$ at $p$, of constant curvature $k_{\alpha
_{p}}=\tanh(r)$ when oriented with the normal pointing to $\Lambda _{k}$. Let 
$d_{\alpha _{p}}$ be the distance function to $\alpha _{p}$ in the connected
component of $\mathbb{H}^{2}\backslash \alpha _{p}$ that contains $\Lambda
_{k}$ and, using \eqref{bar2} set
\begin{equation}
f(s)=\int_{0}^{s}\frac{c_{H}\left( r,t\right)}{\sqrt{1-[c_H(r,t)]^2}} dt\, 
\end{equation}%
$0\leq s:=d_{\alpha _{p}}\left( z\right) \leq z_H\left( r\right) $ and 
\begin{equation}
\overline{\Lambda }_{k,p}=\left\{ z\in \overline{\Lambda }_{k};d_{\alpha
_{p}}\left( z\right) \leq z_H\left( r\right) \right\} .  \label{Lpk}
\end{equation}
Now, define $w_p:\overline{\Lambda }_{k}\longrightarrow \mathbb{R}$ by%
\begin{equation*}
w_{p}\left( z\right) =\left\{ 
\begin{array}{l}
\min \left\{ f( d_{\alpha _{p}}\left( z\right) ),u\left( z\right)
\right\} \text{, if }z\in \overline{\Lambda }_{k,p} \\ 
u\left( z\right) \text{, if }z\in \overline{\Lambda }_{k}\backslash\overline{\Lambda }%
_{k,p}.
\end{array}%
\right. % \label{supwp}
\end{equation*}%

Observe that, if $y\in \partial\Lambda_{k,p}\cap \Lambda_k,$ then $f( d_{\alpha _{p}}\left( y \right) )=f\left( z_H(r)\right)=a_H(r)\geq u(y).$ Hence for any $z\in \Lambda_k,$ there is $U,$ a neighborhood of $z,$ such that either  $w_p\equiv \min \left\{ f\circ d_{\alpha _{p}},u\right\}$ with both functions $f\circ  d_{\alpha _{p}}$ and $u$ well defined in $U$ or $w_p\equiv u$ in $U.$ Since these functions have cmc graphs, we conclude that $w_p$ is a supersolution relatively to the PDE $
Q_{H}\left( u\right) =0  \text{ in }\Lambda_k.$

It follows that, for all $p\in \partial \Lambda_k$, $w_p\geq 0$ in $\overline{%
\Lambda }_{k}$, $w_p\left( p\right) =0$ and therefore $w_p$ is a supersolution relatively to the Dirichlet problem%
\begin{equation}
\left\{\begin{array}{lll}
Q_{H}\left( u\right) =0  \text{ in }\Lambda_k\\
u|_{\partial \Lambda _{k}} =0\\
u\in C^{2}\left( \Lambda _{k}\right) \cap
C^{0}\left( \overline{\Lambda }_{k}\right). 
\end{array}\right.\label{DPL}
\end{equation}
Since the constant function $\phi =0$ in $\overline{\Lambda }_{k}$ is a
subsolution to the Dirichlet problem \eqref{DPL}, from the Perron method,
setting%
\begin{equation*}
S=\left\{ \phi \in C^{0}\left( \overline{\Lambda }_{k}\right) ;\phi \text{
is a subsolution of \eqref{DPL}}\right\} ,
\end{equation*}%
we obtain that%
\begin{equation}
v \left( z\right) =\sup \left\{ \phi \left( z\right) \, |\,\phi
\in S\right\}, \, z\in \overline{\Lambda }_{k}  \label{solLk}
\end{equation}%
is a solution of the Dirichlet problem \eqref{DPL} and, moreover, from the maximum
principle, $\left\vert v \right\vert _{0}\leq a_H(r)-\varepsilon.$

We first suppose $\Omega $ bounded.

Since $\Omega $ is bounded, there is $k\in \mathbb{N}$ such that $\Omega
\subset \Lambda _{k}$. Taking into account such $\Lambda _{k}$, consider a
sequence of bounded $C^{2,\alpha }$ domains $\Omega _{n}$, with $\overline{\Omega
_{n}}\subset \Omega _{n+1}$ and $k_{\partial \Omega _{n}}\geq \kappa $ for
all $n\in \mathbb{N}$, satisfying%
\begin{equation*}
\Omega =\bigcup\limits_{n=1}^{+\infty }\Omega _{n}
\end{equation*}%
Since each $\Omega _{n}$ is included in $\Lambda _{k}$, if $u\in C^{2}\left(
\Omega _{n}\right) $ satisfies $Q_{H}\left( u\right) =0$, $u|_{\partial
\Omega _{n}}=0$, it follows from the maximum principle that $\left\vert
u\right\vert _{0}\leq a_H(r)-\varepsilon.$ Then, by Theorem \ref{MT}, there is $u_{n}\in
C^{2,\alpha }\left( \overline{\Omega }_{n}\right) $ such that $Q_{H}\left(
u_{n}\right) =0$ and $u_{n}|_{\partial \Omega _{n}}=0$. It follows from
standard compactness results (see \cite{GT}) the existence of a subsequence
of $\left( u_{n}\right) $,
converging uniformly on compact subsets of $\Omega$ to $u\in C^{2}\left( \Omega
\right) $ satisfying $Q_{H}\left( u\right) =0$ in $\Omega.$ Besides, from the proof of  Theorem \ref{MT}, there is $M>0,$ such that $\left\vert \nabla u_{n}\right\vert<M$ in $\overline{\Omega_n}$ for all $n \in \mathbb{N}.$ Therefore $u\in C^0(\overline{\Omega})$  and $u|_{\partial \Omega }=0$.

\bigskip

Now, suppose $\Omega $ unbounded.

Consider a sequence of $C^{2}$ bounded domains $\overline{\Omega _{j}}$, $j\in \mathbb{N%
}$, with $k_{\partial \Omega _{j}}\geq \kappa $, such that $\Omega
_{j}\subset \Omega _{j+1}$ and%
\begin{equation*}
\Omega=\bigcup\limits_{j=1}^{+\infty }\Omega _{j},
\end{equation*}%
Notice that, for each $j$ there is $k$ such that $\Omega _{j}\subset \Lambda
_{k}$. Thus, we obtain a subsequence of $\left( \Lambda _{k}\right) $, which
we denominate $\left( \Lambda _{k_{j}}\right) $, with $\Lambda
_{k_{j}}\subset \Lambda _{k_{j+1}}$, such that for each $j$, $\Omega
_{j}\subset \Lambda _{k_{j}}$. Then, according to the bounded case, there is
for each $j\in \mathbb{N}$, a solution $u_{j}\in C^{2}\left( \Omega
_{j}\right) \cap C^{0}\left( \overline{\Omega }_{j}\right) $ of \eqref{DP}. Standart compactness results imply that $\left( u_{j}\right) 
$ has a subsequence, that we name $\left( u_{j}\right) $ again, converging
uniformly on compact subsets of $\Omega $ to a solution $u\in C^{2}\left( \Omega
\right) $ to $Q_{H}=0$ in $\Omega.$ 
By using the barriers as in the proof of Theorem \ref{MT}, we conclude that the norm of the gradient of $u_j$ is uniformly bounded and since
$u_{j}|_{\partial \Omega _{j}}=0$ for all $j$, it follows that $u\in
C^{2}\left( \Omega \right) \cap C^{0}\left( \overline{\Omega }\right) $ and $
u|_{\partial \Omega }=0.$
\end{proof}

\begin{lemma}
\label{LcomMN}Let $H\in \left( 0,1/2\right),$ $r>\tanh^{-1}(2H)$ and $\rho >0$ be given. There are positive numbers $R=R(H,r),$   $\varrho =\varrho \left( H,\rho \right) $ and $\ell
=\ell \left( H\right) $, such that%
\begin{equation*}
h_H(R)=a_H(r),\,
h_H(\varrho)=A_H(\rho)
\text{ and } h_H(\ell)=a_H(\infty)
\end{equation*}%
where $h_H$, $a_H,$ $A_H$ and $a_H(\infty) $ are given by \eqref{hG}, \eqref{hHS}, \eqref{hNod}, and \eqref{hNinf}, respectively. Moreover $\ell$ is the solution of 
\begin{eqnarray}
\ln\left[\left(\frac{\sqrt{1-4H^2}+\sqrt{1-4H^2\tanh^2(\ell)}}{\sqrt{1-4H^2}+1}\right)\cosh(\ell)\right] \nonumber \\
=\frac{\pi \sqrt{1-4H^2} }{4H}-2\tanh ^{-1}\left( \frac{%
1-2H}{\sqrt{1-4H^{2}}}\right). \label{iH}
\end{eqnarray}
\end{lemma}

\begin{proof}
The result follows immediately from the fact that $h_H:[0,\infty)\to [0,\infty)$ in
\eqref{hG} is an increasing homeomorphism.
\end{proof}

\begin{corollary}
\label{C1} Let $H\in (0,1/2)$ and $\Omega \subset \mathbb{H}^{2}$ a $
C^{2} $ domain be given. Set $\kappa =\inf_{\partial \Omega }k_{\partial
\Omega }$. \begin{enumerate}[label=(\arabic{enumi})]

\item \label{corC1-iv}If $\kappa \geq 2H$ and $\overline{\Omega} $ is contained in a region bounded by two hypercycles equidistant to a fixed geodesic, then the Dirichlet problem \eqref{DP} has
a solution in $C^{2}\left( \Omega \right) \cap C^{0}\left( \overline{\Omega }
\right).$
\item \label{corC1-i}If $\kappa \in (-1,2H)$, for $r=\tanh ^{-1}\left( -\kappa \right),$ let $R=R(H,r)$ be as defined in Lemma \ref{LcomMN}. If 
$\overline{\Omega}$ is contained in a region bounded by two hypercycles
equidistant $R$ to a fixed geodesic, then the Dirichlet
problem \eqref{DP} has a solution in $C^{2}\left( \Omega \right) \cap C^{0}\left( \overline{\Omega }\right).$
\item \label{corC1-ii}If $\kappa <-1$, for $\rho =\coth ^{-1}\left( -\kappa \right),$
let $\varrho =\varrho \left(H,\rho \right) $ be as defined in Lemma \ref{LcomMN}. If 
$\overline{\Omega}$ is contained in a region bounded by two hypercycles
equidistant $\varrho$ to a fixed geodesic and $\Omega $ satisfies the exterior circle condition of radius $\rho,$ then the Dirichlet
problem \eqref{DP} has a solution in $C^{2}\left( \Omega \right) \cap
C^{0}\left( \overline{\Omega }\right). $
\item \label{corC1-iii}If $\kappa =-1,$ suppose $\overline{\Omega}$ contained in a region bounded by two hypercycles which are equidistant $\ell $ to a fixed geodesic, where $\ell =\ell\left( H\right) $ is given by \eqref{iH},
then the Dirichlet problem \eqref{DP} has solution in $C^{2}\left( \Omega
\right) \cap C^{0}\left( \overline{\Omega }\right) $.
\end{enumerate}
\end{corollary}

\begin{proof}

The proofs of all cases in this corollary follow the same steps as the proof of Corollary \ref{C3n}. We start by constructing the domains $\Lambda_k,$ $k\in \mathbb{N}.$ Then using the graphs $\mathfrak{G}=\mathfrak{G}\left( l,H\right)$ described in Lemma \ref{dgeo} translated downwards, we construct supersolutions $w_p$ for all $p\in \partial \Lambda_k.$ Then Theorem \ref{MT} implies the existence of a solution $v:\Lambda_k\to \R,$ with $|v|_0<F(\kappa,H),$ $F$ defined in Remark \ref{rmk_FkH}.

Then the same exaustion $\Omega=\cup_n\Omega_n$ described in the proof of Corollary \ref{C3n} must be done. We should only observe that in item  \ref{corC1-ii}, it is possible to assume that each $\Omega_n$ satisfies the same exterior circle condition as $\Omega.$ With this exaustion, the last steps of the previous proof work.
\end{proof}

\begin{flushleft}
\textsc{Departamento de Matem\'atica}

\textsc{Universidade Federal de Santa Maria}

\textsc{Av. Roraima 1000, Santa Maria RS, 97105-900, Brazil}

\textit{Email:}a.aiolfi@gmail.com; patricia.klaser@ufsm.br 
\end{flushleft}

\end{document}